\documentclass[12pt,a4paper]{article}
\usepackage{amsmath,amsthm,amsfonts,amssymb}
\usepackage{lscape}
\usepackage{float}
\usepackage[caption = false]{subfig}
\usepackage[demo]{graphicx}
\usepackage{epstopdf}
\usepackage[mathscr]{euscript}

\usepackage{enumerate}

\usepackage{authblk}

\usepackage{float}
\usepackage{hyperref}
 
\newtheorem{thm}{Theorem}[section]
\newtheorem{cor}{Corollary}[section]
\newtheorem{lem}{Lemma}[section]

\newtheorem{defn}{Definition}[section]
 
\newtheorem{rem}{Remark}[section]

\numberwithin{equation}{section}
 
\begin{document}
\title{Rotation domains and Stable Baker omitted value }

\author[1]{Subhasis Ghora\footnote{sg36@iitbbs.ac.in} }
\author[1]{Tarakanta Nayak\footnote{tnayak@iitbbs.ac.in(Corresponding author )} }
\affil[1]{\textit{School of Basic Sciences\hspace{8cm}Indian Institute of Technology Bhubaneswar, India}  }
\date{}
\maketitle
\begin{abstract}
	 A Baker omitted value, in short \textit{bov} of a transcendental meromorphic function $f$ is an omitted value  such that there is a disk $D$ centered at the bov for which each component of the boundary of $f^{-1}(D)$ is bounded. 
Assuming all the  iterates $f^n$ are analytic  in a neighborhood of its bov, this article proves that the number of Herman rings of a particular period is finite and every Julia component intersects the boundaries of at most finitely many Herman rings. Further, if the bov is the only limit point of the critical values then it is shown that $f$ has infinitely many repelling fixed points. If a repelling periodic point of period $p$ is on the boundary of a $p$-periodic rotation  domain then the periodic point is shown to be on the boundary of infinitely many Fatou components. Under additional assumptions on the critical points, a sufficient condition is found for a Julia component to be singleton. As a consequence, it is proved that if the boundary of a wandering domain $W$ accumulates at some point of the plane under the iteration of $f$ then each limit of $f^n$ on $W$ is either a parabolic periodic point or in the $\omega$-limit set of recurrent critical points. Using the same ideas, the boundary of rotation  domains are shown to be in the $\omega$-limit set of recurrent critical points. 
\end{abstract}
\textit{Keyword:}
Baker omitted value, Recurrent critical point, Rotation  domain and Wandering domain.\\
Mathematics Subject Classification(2010) 37F10,  37F45
\section{Introduction}
A function $f:\mathbb{C} \rightarrow \widehat{\mathbb{C}}$ with exactly one essential singularity, chosen to be at $\infty$ is called  general transcendental meromorphic if it has  either at least two poles or one pole which is not an omitted value (i.e., there is at least one pre-image of the pole).
The Fatou set of $f$ (also called the stable set) denoted by $\mathcal{F}(f)$ is defined as the set of all points at a neighborhood of which the sequence of functions $\{f^n\}_{n \geq 0}$ is defined and normal in the sense of Montel. The complement of the Fatou set is called the Julia set. It is denoted by $\mathcal{J}(f)$.  By the definition of the Julia set, $\infty \in \mathcal{J}(f)$. It is well known that $\{f^n\}_{n>0}$ is normal in a neighborhood of a point whenever $f^n$ is  defined  and analytic  for all $n$ in the neighborhood~\cite{ber93}.

    A point $b\in  \mathbb{C}$ is called an omitted value of $f$ if $f(z)\neq b$ for any $z\in\mathbb{C}$. A special type of omitted value, introduced by Chakra et al. in ~\cite{bov} is the concern of this article. 
 
 \begin{defn}
 A Baker omitted value  $b \in  \mathbb{C}$ of a meromorphic function $f$ is an omitted value of $f$ such that there is a disk $D$ with center at $b$ for which each component of the boundary of $f^{-1}(D)$ is bounded.
 \end{defn} 

 The set of singular values is the closure of the union of the critical values and the asymptotic values. A critical value is the image of a critical point. On the other hand, a point $a\in \widehat{\mathbb{C}}$ is called an asymptotic value of $f$ if there exists a curve $\eta:[0,\infty)\rightarrow \mathbb{C}$ with $\lim_{t\rightarrow \infty} \eta(t)=\infty$ such that $ \lim_{t\rightarrow \infty} f(\eta(t))=a$.  An omitted value is an asymptotic value. 
 % By definition, the singularity lying over a bov is not logarithmic. Though there is only one singularity lying over a bov, 
 The dynamics  (the Fatou and the Julia set) of meromorphic functions, for which the set of singular values is finite or bounded  has been investigated extensively. Some references can be found in~\cite{ber93}. In these studies the nature of the whole set of singular values (finite or bounded) is of primary importance.  Instead of the whole set of singular values, one may look at few singular values of a particular property. How a specific type of singular value influences the dynamics of a function can be a different way to look at the subject. Omitted values are a special type of asymptotic value in the sense that every singularity lying over it is direct. Details on the classification of singularities can be found in ~\cite{berg-ermk}. These are known to control a number of aspects of the dynamics of a function, as reported in ~\cite{tk4,tkzheng,Nayak2016}. A Baker omitted value is always a limit point of critical values (See Lemma~\ref{critically infinite} in the next section). This  gives that the functions with the bov has infinitely many singular values. It is not known whether such a function can have an unbounded set of singular values. The well studied functions $z \to \lambda e^z$ and $z \to \lambda \tan z$ have omitted values. In each of these cases, the pre-image of a sufficiently small neighborhood of an omitted value is connected and simply connected. This fact has been crucial in the investigation in many different ways. But this is not the case for a Baker omitted value. In fact, the pre-image of a sufficiently small neighborhood of a bov is an infinititely connected domain (See Lemma~\ref{singleton}(1)). Thus, the investigation of dynamics of functions with bov is a new direction in transcendental dynamics. This has been initiated in ~\cite{bov}.

 \begin{defn}
 	 The bov of a meromophic function is called stable if the sequence of its iterates is defined in a neighborhood of the bov.  
   \end{defn}
As remarked earlier, the stable bov is in the Fatou set of the function.
 
  Let $\mathcal{M}_S$ denote the class of all general transcendental meromorphic functions with stable bov. 
 
 A  Fatou component is a maximal connected subset of the Fatou set. It is very important to note that  for all functions in $\mathcal{M}_S$, all but one Fatou component are bounded (See Lemma 2.4 (2)~\cite{bov-1} or  Lemma~\ref{singleton} in the next section). This fact has been crucial in ~\cite{bov-1} to determine the connectivity of all the Fatou components.
\par
A Fatou component $V$ is called $p-$periodic if $V_p\subseteq V$ where $V_k$ denotes the Fatou component containing $f^k(V)$ for $k \geq 0$ where $V_0$ is taken as $V$. A periodic Fatou component can be an attracting domain, a parabolic domain, a Baker domain or a rotation  domain (a Herman ring or a Siegel disk). Rotation  domains are special in the sense that  $f^p$ is conformally conjugate to an irrational rotation of an annulus (Herman ring) or the unit disk (Siegel disk) on a $p$-periodic rotation domain. In fact, a $p$-periodic Fatou component $V$ is a Herman ring (or a Siegel disk) if there exists a conformal map $\phi: V \to \{z: 1<|z|<r\}$ (or $\phi: V \to \{z:  |z|<1\}$ respectively) such that $\phi ( f^p( \phi^{-1}))(z)=e^{2 \pi i \theta}z$ for some irrational $\theta$. A rotation domain is an uncountable union of disjoint Jordan curves each of which is invariant under $f^p$. These curves are in deed the pre-images of the concentric circles centered at the origin under $\phi$.  More details  can be found in \cite{ber93}.

\par Meromophic functions with finitely many singular values cannot have infinitely many Herman rings. This is shown by Zheng ~\cite{zheng-2000} who also proved the existence of a function with infinitely many Herman rings. Though a function in $\mathcal{M}_S$ has infinitely many singular values, we have proved that there cannot be infinitely many Herman rings of a particular period. In view of the known result that the period of a Herman ring of a meromorphic function with an omitted value is larger than  two~\cite{Nayak2016}, we consider   Herman rings of period at least three.
\begin{thm}\label{pperiodicHR}
	Let  $f \in \mathcal{M}_S$ and $p \geq 3$. 
	Then the number of $p-$periodic Herman rings is finite.
\end{thm}

Maximally connected subsets of the Julia set  are referred as Julia components. Each component of the  boundary of a Herman ring is always contained in a Julia component. But it is non-trivial to decide whether there can be infinitely many Herman rings sharing their boundaries with a common Julia component. The following result is an answer to this for all the functions with stable bov.

\begin{thm}\label{commonboundarypoint}
	Let $f \in \mathcal{M}_S$ and $J$ be a Julia component of $f$. Then the number of Herman rings whose boundary components are contained in $J$ is finite.

\end{thm}
\begin{cor}
	Let $f \in \mathcal{M}_S$. If the bov is the only limit point of the critical values, then the number of Herman rings whose boundary intersects the forward orbit of a critical value is finite.
\end{cor}
\begin{proof}
	  If the bov is the only limit point of the critical values, the number of critical values lying in the Julia set is finite. Let $c$ be such a critical value whose forward orbit intersects the boundary  of some Herman ring. We shall be done by showing that the forward orbit of $c$ intersects the boundaries of at most finitely many Herman rings. If $\partial H$ is  the component of the boundary of a $p$-periodic Herman ring $H$  then $f^{p}(\partial H) \cap \partial H \neq \emptyset$. To see this, observe that each point of $\partial H$ is a limit point of $f^p$-invariant Jordan  curves in the Herman ring. Further, if $J$ is the Julia component containing $\partial H$ then $f^p (\partial H) \subseteq f^{p}(J) $. 
	Since $f^k (J)$ is connected for all $k$  
	(by Lemma~\ref{singleton}(5)), $f^p (J) \subseteq J$.  Further, $f^{np}(J) \subseteq J$ for all $n$.  Therefore the number of Julia components intersecting the forward orbit of $c$ is finite. Since $J$ intersects the boundaries of at most finitely many Herman rings by the previous theorem, we are done.
\end{proof}

 \par 
 A fixed point $z_0$ of $f$ is called weakly repelling if $|f'(z_0)|>1$ or $f'(z_0)=1$. These are related to the connectedness of the Julia sets. The existence of weakly repelling fixed points for transcendental meromorphic functions are well-known in the presence of multiply connected Fatou components. This is proved for wandering domains in~\cite{ber-weakly} whereas \cite{fag-weakly} deals with the case of    immediate attracting and parabolic domains. These results ensure at least one weakly repelling fixed points. We prove the existence of infinitely many such fixed points for functions with stable bov. 
\begin{thm}
If $f \in \mathcal{M}_S$  then it has infinitely many weakly repelling fixed points. Further, if the bov is the only limit point of the critical values then $f$ has infinitely many repelling fixed points.
	\label{repelling}\end{thm}

The existence of periodic points on the boundary of invariant rotation domains of rational functions is investigated by Imada \cite{periodicptSiegel}. He has proved that the boundary of an invariant  rotation domain does not contain any periodic point except cremer points (i.e., irrationally indifferent periodic point not corresponding to a Siegel disk). Any such result for transcendental meromorphic functions are apparently not known. As a corollary to the following theorem we have shown that  repelling fixed points cannot be on the boundary of invariant rotation domains whenever a function has stable bov and is injective in a neighborhood of the rotation domain. The theorem to follow   says that if a $p$-periodic rotation  domain contains a repelling $p$-periodic point then the topology of the boundary of $D$ is complicated.
\begin{thm}
Let $f\in \mathcal{M_S}$  and $D$ be a $p$-periodic rotation  domain of $f$. If  the boundary $\partial D$ of $D$ contains a repelling $p$-periodic point $z_0$ then, for each $ k \geq 1$ there is a component $D_{-k}$ of $f^{-pk}(D)$  different from $D$ such that $z_0 \in \partial D_{-k}$.
	\label{boundary}
\end{thm}
\begin{cor}
	Under the assumption of the above Theorem, if $f $ is univalent in a neighborhood of one of its invariant rotation  domain $D$ then there is no repelling fixed point on the boundary of $D$.
\end{cor}
Though the functions in $\mathcal{M}_S$ have infinitely many critical values, only finitely many of them can be in the Julia set whenever the bov is the only limit point of the critical values. The next two results  assume that the bov is the only limit point of the critical values. This assumption, already made in Theorem~\ref{repelling} is necessary to make sense of recurrence of singular values in a natural way. 

The $\omega-$limit set $\omega(c)$ of a critical point $c$ is the set of all accumulation points of its forward orbit, i.e., $\omega(c)=\{w: f^{n_k}(c) \to w~\mbox{as} ~k \to \infty~\mbox{for some subsequence}~ n_k\}$. This set is always closed.
A critical point is said to be recurrent if it is in its own $\omega-$limit set.  The possible presence  of either infinitely many critical points or an asymptotic value, or both  makes the study of recurrent singular values difficult for transcendental functions. In fact, the definition of recurrent   singular value itself requires extra considerations.  However, the situation is tractable for functions in $\mathcal{M}_S$  whenever the bov is the only limit point of critical values. The importance of recurrent critical points  is well-known in rational dynamics (See for example~\cite{mane}). The next two results demonstrate  the influence of recurrent critical points on Julia components, wandering domains and rotation  domains.
Wandering Julia component of a rational function are studied by Guizhen et al. in \cite{topologyJC}. They have proved that each such Julia component, except countably many is either non-separating or its complement has two components. Following is a sufficient condition for Julia component of some transcendental functions  to be singleton.

\begin{thm}\label{totallydisconnectedness}
	For $f \in \mathcal{M}_S$ let,
	\begin{enumerate}
		
		\item the bov be the only limit point of critical values,
		\item the number of critical points corresponding to each critical value lying in the Julia set is finite, and
			\item every Fatou component containing a singular value is pre-periodic.
	\end{enumerate} 
	If $J$ is a Julia component whose forward orbit accumulates at a point in $\mathbb{C}$ which is neither a parabolic periodic point nor in the $\omega$-limit set of any recurrent critical point, then $J$ is singleton.    
\end{thm}
Each limit  of the sequence of iterates of a transcendental meromorphic function on its wandering domain is known to be in the derived set of the post singular set~\cite{zheng-singu-2003}. Under the assumption of the above theorem,  $U$ is not wandering giving that its grand orbit cannot contain any wandering domain. Further, every other possible wandering domain of $f \in \mathcal{M}_S$ is simply connected by Theorem 3.1(1) ~\cite{bov-1}. This gives that the boundary of wandering domains is connected. With a condition on the boundary of the wandering domain, we have shown that each  limit of $f^n$ on its wandering domain is a parabolic periodic point or in the 
$\omega$-limit set of recurrent critical points.
\begin{cor}
Let $J$ be the boundary of a wandering domain of a function in $\mathcal{M}_S$ satisfying all the assumptions of Theorem~\ref{totallydisconnectedness}. If  its forward orbit $\{f^{n}(J)\}_{n>0}$ accumulates at a point $w \in \mathbb{C}$ then $w$ is either a parabolic periodic point or is in the $\omega-$limit set of recurrent critical points.
\end{cor}
\begin{proof}
Each wandering domain of $f$ is simply connected and so $J$ is connected and not singleton. By Theorem~\ref{totallydisconnectedness}, $w$ is either a parabolic periodic point or is in the $\omega-$limit set of recurrent critical points.
 
\end{proof}
The forward orbit of recurrent critical points are known to be dense in the boudary of rotation  domains of rational maps (See for example ~\cite{mane}). The main idea was to show a kind of backward contraction of pull backs of disks disjoint from the $\omega$-limit sets of recurrent critical points. The same conclusion is obtained for functions with stable bov with some other assumptions.
\begin{thm}
	Let  $f \in \mathcal{M}_S$ and  its bov be the only limit point of its critical values and the number of critical points corresponding to each critical value is finite. If each critical value belonging to the Fatou set is contained in  some pre-periodic Fatou component then the boundary of each rotation  domain is contained in the $\omega-$limit set of the recurrent critical points.
	\label{rotationaldomainboundary}\end{thm}

Throughout this article, the Fatou component containing the bov is denoted by $U$. The disk $\{z: |z-a| < \delta\}$ is denoted by $D_{\delta}(a)$ for  $\delta>0$ and $a \in \mathbb{C}$. All the functions considered in th earticle are in $\mathcal{M}_S$ unless stated otherwise.

\section{Preliminary results}

 A non-empty connected and closed subset of $\widehat{\mathbb{C}}$ is called a continuum. A continuum $K$ is called  full if $\widehat{\mathbb{C}} \setminus K$ is connected. If $K$ does not contain the point at $\infty$ and is not full then its complement has at least one bounded component. We say $K$ surrounds a point $z$ if a bounded component of $\widehat{\mathbb{C}} \setminus K$ contains $z$. Clearly a full continuum does not surround any point.  A Julia component is called full if it is a full contiuum. Otherwise, it is called non-full. Non-full continua are also called separating. The following lemma  is a generalization of Lemma 1 (\cite{tkzheng}) and is to be used frequently. The proof is essentially the same. We require a  lemma  from~\cite{tk4}.
 \begin{lem}\label{bounded}
If $f$ is a meromorphic functions with an omitted value and $D$ is a bounded domain then the closure of $f(D)$ cannot contain any omitted value.
\end{lem}

\begin{lem} Let $f $ be a transcendental meromorphic function and $K$ be a separating continuum  not intersecting the backward orbit of $\infty$. If $K$ surrounds a point  in $\mathcal{J}(f) \setminus K$ then there is an $n \geq 0$  such that $f^n (K)$ surrounds a pole of $f$. Further, if $f$ has an omitted value then $f^{n+1}(K)$ surrounds the set of all omitted values of $f$.
\label{general1}
 \end{lem}
\begin{proof}
Consider a point $z \in \mathcal{J}(f) \setminus K$ which is surrounded by $K$. Then there is a component $V$ of $\widehat{\mathbb{C}} \setminus K$ containing $z$. Note that there is an $n $ such that $f^{n}$ is analytic on $V$ and $f^{n}(V)$ contains a pole, say $w$. By the Maximum Modulus Principle, $f^{n}(K)$ surrounds $w$. The set $f^{n+1}(V)$ contains a neighborhood of $\infty$ and by Lemma~\ref{bounded},  the closure of $f^{n+1}(V)$ does not contain any omitted value. This gives that $f^{n+1}(K)$ surrounds the set of all omitted values of $f$.

\end{proof}
% \textbf{If a separating Julia component $J$ does not intersect the backward orbit of $\infty$ then $f^k(J)$ is separating... and surrounds a point of the Julia set then it follows from Lemma~\ref{general1} that  and surrounds the bov for infinitely many values of $k$.
% }
It is well-known that for a meromorphic function, each pre-image component of every neighbourhood of an omitted value is unbounded. Here is another simple but useful observation. 
\begin{lem}\label{nov}
	Let $f$ be a meromorphic function with an omitted value. If $V$ is a $p$-periodic component of $f$ and $f^p:V \rightarrow V$ is one-one then $V$ does not contain any omitted value of $f$. In particular,  rotation  domains donot contain any omitted value.
\end{lem}

\begin{proof}
	
	On the contrary, suppose that $V$ contains an omitted value $b$ of $f$. Let $B_{\epsilon}(b)$ be a ball centered at $b$ with radius $\epsilon$. Note that each component $B_{-1}$ of $f^{-1}(B_{\epsilon}(b))$ is unbounded for each $\epsilon $, and $f:B_{-1} \rightarrow \mathbb{C}$ is not one-one. Choosing $\epsilon > 0$ such that $B_\epsilon(b)\subset V$, we have $f:V_{-1} \rightarrow V$ is not one-one, where $V_{-1}$ is the periodic pre-image of $V$ under $f$. Thus $f^p:V \rightarrow V$ is not one-one. This leads to a contradiction.
\end{proof}
 We put 
Theorem 2.1 and Lemma 2.3(1) of ~\cite{bov-1} together as a lemma that exhibits the influence of the bov on all other singular values of the function.
\begin{lem}\label{critically infinite}
	Let $f $ be a meromorphic function with bov.
	Then, 
\begin{enumerate}
\item The bov  is a limit point of its critical values.
\item The function $f$ has only one asymptotic value and that is the bov.
\end{enumerate}
\end{lem}
Some useful observations on the dynamics of functions with stable bov are made in Lemma 2.4 and Theorem 3.3 of ~\cite{bov-1}. We collect them here.
\begin{lem}\label{singleton}
Let $f \in \mathcal{M}_{S}$ and $U$ be the Fatou component containing the bov. Then, 
\begin{enumerate}

\item The pre-image of $U$ is the only unbounded  Fatou component of $f$. Further, it is infinitely connected.
%\item The unbounded Fatou component $U_{-1}$ is completely invariant if and only if it contains the bov.
\item If $U'$ is a Fatou component such that $U'_{k} =U$ for some $k \geq 1$ then $U'$ is infinitely connected.
\item If $U$ is invariant then it is completely invariant.
\item  There are infinitely many poles of $f$.
\item All the components of $\mathcal{J}(f) \cap \mathbb{C}$ are bounded. In other words, every Julia component intersecting the backward orbit of  $\infty$ is singleton.
Consequently, for every non-singleton Julia component $J$, $f^k (J)$ is connected for all $k \geq 1$. 
\item If the Fatou component containing the bov  is unbounded then it is completely invariant. Consequently, $f$ has no Herman ring.

\end{enumerate}
\end{lem}  

We say a set $A$ surrounds another set $S$ if a bounded complementary component of $A$ contains $S$. 
The following two definitions appearing in ~\cite{relevantpole} are very important in analyzing the arrangement of Herman rings.
For a Herman ring $H$, $H_n$ denotes the Herman ring containing $f^n (H)$.
\begin{defn}[Innermost ring with respect to a set]
Given a Herman ring $H$, we say  $H_k$ is innermost with respect to a set $A$ if $H_k$ surrounds $A$ but does not surround any $H_i$  for  $ i \neq k$.
\end{defn}
 By taking $K$ as an $f^p$-invariant Jordan curve in a $p$-periodic Herman ring $H$ in Lemma~\ref{general1}, it is observed that $H_n$ surround a pole for some $n$. We choose $n$ to be the smallest such number. Further, if $H$ is taken to be the innermost Herman ring with respect to the bov then the first $n$ forward iterates of $H$ turns out to be crucial.
\begin{defn}
[Basic chain]Given a Herman ring $H$, the ordered set of rings $\{H_1, H_2, H_3,\dots, H_k\}$ is called the basic chain, where $H_1$ is the innermost ring with respect to the bov and $k$ is the smallest natural number such that $H_{k}$ surrounds a pole. The number $k$ is called the length of the basic chain.
\end{defn}
 The basic chain of a cycle is also referred as the basic chain of a Herman ring contained in the cycle. 
 
We collect some known facts about basic chains. Recall that $U$ denotes the Fatou component containing the bov. A finite sequence of rings $\{H_{j}, H_{j+1}, H_{j+2},\cdots, H_{j+m}\}$   is called a chain whenever $H_j$ surrounds $U$ but not any pole and $m$ is the smallest natural number such that $H_{j+m}$ surrounds a pole. 
\begin{lem}\label{basicchain}
	Let $f$ be a meromorphic function having a bov.
	\begin{enumerate}
		\item Every cycle of Herman rings has a unique basic chain.
		\item The length of every chain is less than or equal to that of the basic chain.
	 
		\item For every $p-$cycle of Herman rings, the length of the basic chain $l_C$ satisfies $2 \leq l_C \leq  p-1$.
	\end{enumerate}
\end{lem}
\begin{proof}
	\begin{enumerate}
\item  This is evident from the definition of the basic chain.

\item This is  Lemma 2.3 of ~\cite{relevantpole}.

\item Since the innermost ring $H_1$ with respect to the bov never surrounds a pole by Remark 2.10 of \cite{small}, the length of the basic chain is at least two. 
If the length of the basic chain is equal to the period of the Herman ring then it follows from the definition of the basic chain that there is only one $H_1$-relevant pole, i.e., the total number of distinct poles surrounded by any of the Herman rings of the cycle is $1$. However the number of $H$-relevant poles of every Herman ring of a function with an omitted value is at least two by Lemma 2.11,~\cite{small}. Hence the length of the basic chain corresponding to a $p$-cycle of Herman rings is at most $p-1$.
	\end{enumerate}
\end{proof}
 We continue to reveal the connection of $U$, the  Fatou component containing the bov  to the possible Herman rings of the function.
For a cycle of Herman rings $C$ with  $l_C$ as the length of its basic chain, let $\mathcal{S}_C$ denote the set $\{U_1, U_2, U_3, \cdots, U_{l_C}\}$ where $U_1 =U$.

\begin{lem}\label{finiteBis}
	Let $f \in \mathcal{M}_S$ and $U$ be the Fatou component containing the bov.
	\begin{enumerate}
		\item If $H$ is a $p$-periodic Herman ring and the length of its basic chain is $l$ then for each $n$ there is an $i \in \{1, 2,3,\cdots, l\}$ such that  $H_n$ surrounds $U_i$ where $U_1=U$.
		\item Let $C'$ and $C''$ be two $p$-cycles of Herman rings. If $l_{C'} >l_{C''}$, $l_{C'} =l_{C''}$ or $ l_{C'} < l_{C''}$  then $\mathcal{S_{C'}} \supsetneq \mathcal{S_{C''}} $, $\mathcal{S_{C'}} = \mathcal{S_{C''}}$ or $\mathcal{S_{C'}} \subsetneq \mathcal{S_{C''}} $  respectively.

	\end{enumerate}

\end{lem}
\begin{proof} 
	\begin{enumerate}
		\item 
		
		Let $\{H=H_1,~H_2,..,H_p\}$ be a $p-$cycle  of Herman rings and $H_1$ be the innermost ring with respect to $U$. Since the length of the basic chain is $l$,  $H_i$ surrounds $U_i$  for $i=1,2,...,l$. Further, $H_l$ surrounds a pole by the definition of the basic chain. By Lemma~\ref{general1}, $H_{l+1}$ surrounds  the bov. But the bov is in $U=U_1$ which is not a Herman ring by Lemma~\ref{nov}. Hence $H_{l+1}$ surrounds $U_1$. 
		
		If $H_{l+1}$ surrounds a pole then there is an $l' \geq 1$ such that $H_{l+1+l'}$ surrounds $U_1$ but not any pole. If $k$ is the smallest natural number for which $H_{l+1+l'+k}$ surrounds a pole then it follows from the Maximum Modulus Principle  that $H_{l+1+l'+j}$ surrounds $U_j$ for all $j  \leq k$. Note that $\{H_{l+1+l'+j}: 1 \leq j \leq k\}$ is a chain.  By  Lemma~\ref{basicchain}(2), the basic chain is the longest chain. In other words,  $k \leq l$. Now $H_{l+1+l'+k+1}$ surrounds $U_1$ by Lemma~\ref{general1}. This argument can be continued with $H_{l+1+l'+k+1}$ instead of $H_{l+1}$ for finitely many times to complete the proof.
		
		\item  Each cycle of Herman rings contains a ring which is the  innermost with respect to the bov. Let $H_1$ and $G_1$ be such innermost rings of the $p$-cycles $C'$ and $C''$ respectively. Both $H_1$ and $G_1$ surround $U_1$. Using Lemma \ref{finiteBis}(1), we have $\mathcal{S_{C'}}=\{U_1  ,~U_2,~U_3,\dots,~U_{l_{C'}}\}$ and $\mathcal{S_{C''}}=\{U_1,~U_2,~U_3,\dots,~U_{l_{C''}}\}$. The proof now follows.

	\end{enumerate} 
	
\end{proof}

\section{Proofs of the results}
Here is the proof of the first result of this article.

\begin{proof}[Proof of Theorem~\ref{pperiodicHR}]
If $f$ has no Herman ring then there is nothing to prove. Else $U$, the Fatou component containing the bov  is bounded by Lemma~\ref{singleton} (6)).  Further  $f^{-1}(U)$ is unbounded, infinitely connected and all its complementary components are bounded by Lemma \ref{singleton}(1). We write $ \mathbb{C} \setminus f^{-1}(U) =\cup_{i=1} ^{\infty} B_i$ and choose $B_1$  such that it contains $U_1$.
\par 
 Suppose on the contrary that $\{C_n\}_{n>0}$ is the set of all the $p-$cycles of Herman rings.  Let $l_n$ be the length of the basic chain of $C_n$. Since $ l_n \leq p-1$ for all $n$ by Lemma~\ref{basicchain}(3), $\max \{l_n: n >0\} =l$ is a finite number. If $C$ is a $p-$cycle of Herman rings with $l_C =l$ then  $ \mathcal{S}_{C_n} \subseteq \mathcal{S}_{C}  $ for all $n$ by Lemma \ref{finiteBis}(2). Recall that $\mathcal{S}_C$ denotes the set $\{U_1, U_2, U_3, \cdots, U_{l_C}\}$ where $U_1 =U$. Here more than one cycle with maximum length of basic chain  are not ruled out and for each such $C$, $\mathcal{S}_C$ is the same set. Let $K=\cup\{B_i:~B_i~\mbox{contains at least one element of}~\mathcal{S}_{C}\}$. Every $p-$periodic Herman ring belongs to  $C_n$  for some $n$ and therefore surrounds an element of $\mathcal{S}_{C_n}$ by Lemma~\ref{finiteBis}. Since every Herman ring is different from $f^{-1}(U_1)$ and $\mathcal{S}_{C_n} \subseteq \mathcal{S}_{C} $, every $p-$periodic Herman ring is in $K$. In other words, for each $p$ there is a compact set $K$ such that all the $p-$periodic Herman rings are contained in $K$.
 \par

%Note that $H^n_1 \subset B(H^m_1)$. If not then let $H^m_1 \subset B(H^n_1)$. Note that by the definition of the basic chain, for $i=1,~2,\dots,~l_m$, both $H^m_i$ and $H^n_i$ do not surround any pole of $f$ giving that $f^i : B(H^m_i) \rightarrow B(H^m_{i+1}) $ is analytic.  Thus, for $i=1,~2,\dots,~l_m$, $H^m_i \subset B(H^n_i)$. Since $H^n_{l_m}$ surrounds a pole of $f$, $H^m_{l_m}$ also surrounds a pole giving that  $l_n \leq l_m$.
%
%Thus for any two arbitrary Herman ring $H^n_1$ and $H^m_1$ with  $l_n \geq l_m$, $H^n_i$ and $H^m_i$ are contained in the same $B_k(i)$ for $i=1,~~2,~\dots l_m$, where $k(i)$ may not be distinct but less than or equals to $l_m$. Note that for a given $p$, the length of the basic chain is bounded $p-1$ and hence $k(i)$ is also bounded by $p-1$. Thus for a given $p$, using Lemma \ref{finiteBis} we have the forward images of $\{H^j\}$ are contained in atmost $p-1$ many $B_i$s.

Each $C_n$ contains a ring which is innermost with respect to the bov. Let such a ring be denoted by $H^n _1$. If required, after passing to a subsequence we can find a sequence of innermost rings $\{H_1 ^{n}\}_{n>0}$ such that either $H_{1}^{n+1}$ surrounds $H_{1}^{n}$ for all $n$, or $H_{1}^{n+1}$ is surrounded by $H_{1}^{n}$ for all $n$. Let $A_n$ be a topological annulus bounded by two $f^p -$invariant Jordan curves, one contained in $H_1 ^{n}$ and the other in  $H_{1}^{n+1}$. Without loss of generality we assume that $A_i \cap A_j = \emptyset$ for $i \neq j$.  Observe that $A_n \subset K$ for all $n$.

\par  If $f^p: A_n \to \mathbb{C}$ is analytic then $f^{kp}(A_n)=A_n$ for all $k$ and  $\{f^{kp} \}_{k>0}$ becomes normal in $A_n$. But this is not possible as $A_n$ intersects the Julia set of $f$. Hence   $f^p$ has a singularity in $A_n$. Each such singularity $z$ must satisfy $f^k (z)=\infty $ for some $1 \leq k \leq p$. Since no innermost ring (with respect to the bov) surrounds  a pole, each singularity of $f^p$ in $A_n$ must be an element of  $\{z: f^k (z)=\infty, 1 <  k \leq p\}$.  This means that there is an integer $k'$, $1 < k' \leq p$ such that  $f^{k'}$ has a pole $w_n$ in $A_n$ for infinitely many values of $n$. Since $w_n \in B_1$ for all $n$, $\{w_n\}_{n>0}$ has a limit point say $w$. Without loss of generality assume $w_n \to w$ as $n \to \infty$. Thus $w$ is an essential singularity of $f^{k'}$. In other words,  there exists an $l <k'$ such that $f^l(w) = \infty$.  Hence $f^l(w_n) \rightarrow \infty$ as $n \to \infty$. The line segment  joining $w_n$ and $w_{n+1}$ contains at least one point $z_n$ of $H_{1}^n$. Since each $A_n$ surrounds the bov and $w_n \to w$, $z_n \to w$ as $n \to \infty$. This implies that $f^l(z_n) \rightarrow \infty$ as $n \to \infty$. In other words, there is an unbounded sequence of $p-$periodic Herman rings, namely $\{f^l (H_1 ^n) \}_{n>0}$.   But this is not possible as all such rings are contained in the bounded set $K$. Thus the number of $p-$cycles of Herman rings is finite.

\end{proof}

\begin{rem}
	In the proof, it is important to note that $K$ contains all the Herman rings whose length of the basic chain is $l$ irrespective of the periods of the Herman ring. This gives rise to a slightly more generalized version of Theorem~\ref{pperiodicHR}; for a given $l$, the number of Herman rings (of a function with stable bov), the length of whose basic chains is $l$,  is finite.
\end{rem}

  We need a definition to prove Theorem~\ref{commonboundarypoint}.
The outer (or inner) boundary of a Herman ring $H$ is the boundary of the unbounded (bounded respectively) component of $\widehat{\mathbb{C}} \setminus H$. By saying a Herman ring surrounds a pole we mean that its inner boundary surrounds the pole.

\begin{proof}[Proof of Theorem~\ref{commonboundarypoint}]
 	 
	First we make two useful observations on   the Julia components meeting the boundary of a Herman ring. Let $J$ be a Julia component intersecting the boundaries (inner or outer) of more than one Herman ring. Note that a Julia component cannot intersect both the boundaries of a Herman ring.
	\begin{enumerate}
		\item  If  two boundary components of two Herman rings intersect $J$ then both  of these   cannot be the inner boundaries of the respective Herman rings.  In other words, all the boundary components  of Herman rings contained in $J$ are outer, with a possible exception. 
	\label{allouter} 
	 \item Two Herman rings whose outer boundaries are contained in a  Julia component cannot surround a common pole. As mentioned earlier, this means that the inner boundaries of these Herman rings cannot surround the same pole.
	 \label{nonnested}
	\end{enumerate}
%	 \begin{equation}~\mbox{ The boundary of each Herman ring, except possibly one contained in }~\mathcal{J}~\mbox{ is outer} 
%	\label{allouter}.\end{equation}
% 	\begin{equation}~
% 	\mbox{ all the inner boundaries of}~H^n,~\mbox{ except possibly one are strictly non-nested. }
% 	\label{nonnested}\end{equation}
In order to prove this theorem by the method of contradiction, suppose that $J$ contains the boundary components of infinitely many Herman rings. By the observation~(\ref{allouter}) above, all except possibly one such boundary components are outer. Let $\{J_n\}_{n>0}$ be the sequence of such outer boundaries of distinct Herman rings $H^n$. Since $J$ is bounded, the number of all poles surrounded by some sub-continuum of $J$ is finite. By the observation  (\ref{nonnested}) above, the Herman ring $H^n$ does not surround any pole for infinitely many values of $n$. Without loss of generality we assume that $H^n$ does not surround any pole for any $n$.
\par 
For each $n$, there is a $k_n$ such that $f^{i}(H^n)$ does not surround any pole for $0 \leq i< k_n$ but $f^{k_n}(H^n)$ surrounds a pole, by Lemma~\ref{general1}. The outer boundary of $f^{k_n}(H^n)$ is the $f^{k_n}-$image of the outer boundary of $H^n$ by the Maximum Modulus Principle. Since each point of the backward orbit of $\infty$ is a singleton Julia component, and in particular does not intersect the boundary of any Herman ring, $f^{k}(J)$ is bounded for each $k$.  If $p =\min \{p(n): H^n~\mbox{is}~p(n)\mbox{-periodic}\}$ then  $p \geq 3$ and $f^p(J) \subseteq J$. Consequently there is a Julia component $J^* \in \{J, f(J), f^2(J),\cdots, f^{p-1}(J)\}$ containing the outer boundary of $f^{k_n}(H^n)$ for infinitely many values of $n$. Since $J^*$ is bounded, the number of all poles surrounded by any of its sub-continuum is finite. Hence one pole is surrounded by at least two Herman rings (their inner boundaries). But  this is not possible by the observation~(\ref{nonnested}), leading to a contradiction.  

\end{proof}
We now present the proofs of Theorem~\ref{repelling} and ~\ref{boundary}.
\begin{proof}[Proof of Theorem~\ref{repelling}]
	Let $U$ be the Fatou component containing the bov and $U_{-1}=f^{-1}(U)$ be the pre-image of $U$. It follows from Lemma \ref{singleton} that $U_{-1}$ is infinitely connected and all its complementary components are bounded. Let $\{B_i\}_{i=1}^\infty$ be the set of all such components. Note that functions with a bov has infinitely many poles, otherwise $\infty$ will be an asymptotic value, which is not true as bov is the only asymptotic value (Lemma \ref{critically infinite}) of the function. Thus $B_j$ contains at least one pole for infinitely many values of $j$.
	Let $ \{\gamma_j\}_{j>0} $ be an  infinite sequence of Jordan curves in $U_{-1}$ such that $ \gamma_j $ surrounds $B_j$ but not any other complementary component of $U_{-1}$. Note that no $\gamma_j$ contains any pole  and at most one $\gamma_j$ surrounds the bov. Then by Corollary $2.9$, \cite{Baranski}, $f$ has infinitely many weakly repelling fixed points. Since bov is the only limit point of the critical values, there are finitely many critical values which qualify to be in some parabolic domain. This gives that $f$ has at most finitely many parabolic periodic points, and in particular finitely many parabolic fixed points. Hence $f$ has infinitely many repelling fixed points.
\end{proof}

The main idea of the next proof is that the branch of $f^{-p}$ at a repelling $p$-periodic point of $f$ is different from the branch of $f^{-p}$ fixing the rotaion domain which contains the periodic point on its boundary. 

\begin{proof}[Proof of Theorem~\ref{boundary}]
	Since $z_0$ is a repelling $p$-periodic point, $|(f^p)'(z_0)|>1$.  Let $g$ be the inverse branch of $f^p$ defined on a neighborhood $N$ of $z_0$ such that $g(z_0)=z_0$. Since $|(f^{-p})'(z_0)|<1$ , $z_0$ be an attracting fixed point of $g$. Let $N' \subset N$ be a neighbourhood of $z_0$ such that  $g(N')\subset N'$. The existence of such $N'$ is evident from the fact that  analytic functions are locally conformally conjugate to linear maps at their attracting fixed points. As $z_0 \in \partial D$, there exists $z \in D \cap  N' $. Define $z_{n}$=$g^n(z)$ for $n \geq 1$ such that $z_n \in N'$. Clearly $z_{n} \rightarrow z_0$ as $n \rightarrow \infty$.  Let $\gamma$ be an $f^p -$invariant Jordan curve contained in $ D$ which contains $z$.  If $z_n \in D$ for infinitely many values of $n$ then each such $z_n$ must be on $\gamma$. But $\gamma$ is at a positive distance from $z_0$ contradicting $z_n \to z_0$ as $n \to \infty$. Thus, there is a natural number $n_0$ such that $z_{n_0} \in D$ but   $z_{n} \notin D$ for all $n > n_0$. Now $z_{n_0 +1} \in g(D)$ and $g(D) \cap D =\emptyset$. If $g^{2}(D) \cap g(D) \neq \emptyset$ or $g^{2}(D) \cap  D  \neq \emptyset $ then applying $f^p $ on these sets we get $D \cap g(D) \neq \emptyset$,  which is just shown to be impossible. Inductively, it can be shown that $g^{k}(D) \cap g^{i}(D) =\emptyset$ for all $0 \leq i \leq k-1$. Denoting  $g^{k}(D)$ by $D_{-k}$, it is seen that $z_0 \in \partial D_{-k}$ for all $k$. The proof completes.  
\end{proof}

 The proof of Theorem~\ref{totallydisconnectedness} is based on some ideas developed in~\cite{mane}.    We start by stating two lemmas proved in the same paper. We replace the unit disk by $D_{R}(a)$ and $D_r (0)$ by $D_{rR}(a)$ in the original form of Lemma 2.1 (\cite{mane}). This is not any loss of generality. A hyperbolic domain is an open connected subset of $\widehat{\mathbb{C}}$ whose complement contains at least three points. If a non-constant function $h$ is analytic at a point $z_0$ then it is locally conjugate to a monomial $z \mapsto z^m$ for some $m >0$ at $z_0$. This $m$ is known as the local degree of $h$ at $z_0$. If $m>1$ then $z_0$ is a critical point of $h$ and its multiplicity is $m-1$.  For two domains $A$ and $B$, a map $h: A \to B$ is called proper of degree $k$ if  for each $b \in B$ the number of pre-images of $b$ in $A$ counting multiplicities is $k$.
\begin{lem}\label{diameterzero}
	For every natural number $d$ and $r \in (0,1)$, there exists $C(d,r)>0$ such that for a given simply connected hyperbolic domain $V$ and a proper analytic map $g:V \rightarrow D_{R} (a)$ of degree at most $d$, each component of $g^{-1}(\overline{D_{rR}(a)})$ has diameter less than $C(d,r)$ with respect to the hyperbolic metric of $V$. Moreover, $ lim_{r\to0} \text{C(d,r)}=0$.
\end{lem}
For a hyperbolic domain $A$ and $A' \subset A$, the  diameter of $A'$ with respect to the hyperbolic distance of $A$ is denoted by $diam_{A} A'$. Let $|A|_s$ denote the spherical diameter of $A$.
\begin{lem}\label{sphericaldiameter}
	Let $W$ be a simply connected domain and $a \in W' \subset W  \subset \Omega \subset \mathbb{C} \setminus \{0\}$ for two domains $W'$ and $\Omega$. If  $diam_{W} W' \leq C$ then   $|W'|_s \leq 2 C'(e^{2C}-1)   \inf \{d(a, \partial \Omega),\frac{1}{|a|}\}$ where $C'$ is a universal constant.
\end{lem}
The next lemma deals with the pre-image component of simply connected domains containing exactly one critical value under proper maps. Though the proof seems to be well-known, no reference is known to the authors and a proof is given.
\begin{lem}\label{simplyconnectedlema}
	Let $h:A \rightarrow B$ be a proper analytic map such that $B$ is simply connected and contains at most one critical value of $h$, then $A$ is simply connected. 
\end{lem}

\begin{proof}
	First note that, by the Riemann-Hurwitz formula (Theorem 5.4.1, ~\cite{Beardon}), $c(A)-2=d(c(B)-2)+N$, where $d$ is the degree of $h$ and $N$ is the number of critical points of $h$ in $A$ counting multiplicity. Since $B$ is simply connected, $c(A)=2-d+N$.
	If $B$ does not contain any critical value then $A$ does not contain any critical point and $h$ is one-one. In other words, $A$ is simply connected. If $B$ contains a critical value then it is the only critical value of $h$ in $B$, by the hypothesis of this lemma. This gives that $c(A)=2-d+N$. If $A$ contains only one critical point then the local degree of $h$ at the critical point is $d$ and hence $N=d-1$. Thus $c(A)=1$. Note that $A$ cannot contain more than one critical point, because that would give $ d - N \geq 2$ (as  all the critical points correspond to the same critical value and the sum of the local degrees at all these critical points is $d$). Consequently $c(A) \leq 0$, which is impossible.
\end{proof} 

Now we proceed to prove Theorem \ref{totallydisconnectedness}. For a proper map $g: A \to B$, let $\deg(g: A \to B)$ denote its degree.

\begin{proof}[Proof of Theorem~\ref{totallydisconnectedness}] Let $f$   satisfy all the hypotheses of the theorem. Suppose that $w \in \mathbb{C}$ is an accumulation point of $J$, i.e., there exists a sequence $n_k \in \mathbb{N}$ such that $\lim_{k \to \infty}f^{n_k}(J)=w$. Here, the limit is with respect to the Hausdorff metric. 
\par 
Let $U$ be the Fatou component containing the bov. Let $ K^0$ be the closure of a simply connected subset of $U$ containing all the critical values belonging to $U$ and the bov. By the first assumption, the number of   Fatou components different from $U$ and  containing some critical value is finite. Let $\{U^i, 1 \leq i \leq N \}$ be the set of all such  Fatou components. Consider the closure $K^i$ of a simply connected domain in $U^i$ containing all the critical values in $U^i$. Set \begin{equation} B = \overline{\cup_{i=1}^{N} \cup_{k \geq 0} f^k(K^i)}.
\label{B}\end{equation} 
Then $B$ is a forward invariant  compact subset of the union of $\mathcal{F}(f)$ and $  \{z: z ~\mbox{is a parabolic periodic 
	 point of }~f\}$. It is important to note here that $f \in \mathcal{M}_S$ has no Baker domain by Theorem~3.4 ~\cite{bov-1}.
 
 Note that $\widehat{\mathbb{C}} \setminus B$ is a backward invariant open set containing all the critical points of $f$ belonging to the Julia set. The critical points belonging to the Fatou set may be in $\widehat{\mathbb{C}} \setminus B$, but this does not matter when we discuss pull backs of disks centered at a point of the Julia set that is not a parabolic periodic point. More precisely, the forward orbits of these critical points cannot accumulate at any point of the Julia set.  In view of the assumption, let $NC=\{c_1, ~c_2,\dots ~c_k\}$ be the set of all non-recurrent critical points of $f$ in the Julia set. Also, let $deg(f,c_i)$ denote the local degree of $f$ at $c_i$ and  $d = \prod_{i=1}^{k} deg(f,c_i)$. Let  $C_1=N_0 dC(d,\frac{2}{3})$ where $C(d, \frac{2}{3})$ is the constant as defined in the Lemma \ref{diameterzero} and $N_0 -1$ is the number of open disks $D_{\frac{1}{3}}(z),  \frac{2}{3}< |z| < 1 $ whose union covers $\{z: \frac{2}{3} \leq |z|\leq 1\}$. Corresponding to each $c_i$, choose a repelling periodic point $w_i$ sufficiently close to $c_i$ such that the cycle of $w_i$ does not contain $w$, the accumulation point of $J$ and such that the set $\Omega = \widehat{\mathbb{C}} \setminus \{B \cup_{i=1}^{k}\{f^n(w_i): n \geq 1 \}\}$ satisfies the following conditions.
	
\begin{enumerate}
\item $d_\Omega(c_i,f^n(c_i)) \geq C_1$ for all $n  \geq 1$ and $1 \leq i \leq k$,
\item $d_\Omega(f(c_i),f(c_j)) \geq C_1 $ whenever $f(c_i) \neq f(c_j)$,
where $d_\Omega$ is the hyperbolic distance of $\Omega$.
\end{enumerate} 
Such a choice of $\Omega$ is possible as the set of all repelling periodic points is dense in the Julia set.  
Since $\Omega$ does not contain any asymptotic value of $f$ (by Lemma~\ref{critically infinite}), for every simply connected domain $D \subset \Omega$ and every component $D'$ of  $f^{-1}(D)$, 
$f: D' \to D$ is a proper map. Considering a conformal conjugate of $f$, if necessary we assume that $0, \infty \notin \Omega$. Note that $\Omega$ is a hyperbolic domain containing $\mathcal{J}(f) \setminus  \{z: z ~\mbox{is a parabolic periodic 
	point of }~f\}$. In particular, $ w\in \Omega$.
\par
Let $w$ is neither a parabolic periodic point nor an accumulation point of any critical point. Then there exists   $D_r = \{z:|z-w|<r\}$ and  $D_{2r} = \{z:|z-w|<2r\}$   such that  $diam_{\Omega}D_{2r} <  C_1$ and $D_{2r}$ does not intersect the $\omega$-limit set of any recurrent critical point or any parabolic periodic point.  Clearly, $diam_\Omega(D_{r}) < C_1$.

We claim that, for all $n$ and  for every connected component $ V'_n $ of $f^{-n}(D_r)$,	
\begin{align}\label{claim}	V'_n ~\mbox{ is simply connected}, 
deg(f^n:V'_n \rightarrow D_r) \leq d, ~\mbox{ and}~\\ diam_\Omega V'_n < C_1.
\end{align}
To prove it, we proceed by induction on $n$.
\par	
Let $V'_1$ and $V_1$ be the components of $f^{-1}(D_r)$ and $f^{-1}(D_{2r})$ respectively such that $V'_1 \subset V_1$. By the choice of $\Omega$, $D_{2r}$ (and hence $D_r$) contains at most one critical value. It follows from Lemma \ref{simplyconnectedlema} that $V'_1$ and $V_1$ are simply connected. Now  $f : V'_1 \rightarrow D_{r}$ is a proper map of degree at most $d$  (It is in fact $1$ if $D_r$ does not contain any critical value). Choose $z_1,~z_2,\dots,~z_{N_0 -1} \in  A(w;\frac{2r}{3}, r)=\{z: \frac{2r}{3} < |z-w| < r\}$ such that $  \overline{A(w;\frac{2r}{3}, r)} \subset \cup_{i=1}^{N_0 -1} D_{\frac{r}{3}}(z_i)$. 
	Since $D_{\frac{r}{2}}(z_i) \subset D_{2r}$,  $D_{\frac{r}{2}}(z_i)$ contains at most one critical value. It follows from Lemma \ref{simplyconnectedlema} that each component $\tilde{ V}$ of $f^{-1}(D_\frac{r}{2}(z_i))$  is simply connected. By the choice of $\Omega$ , $ f:\tilde{ V} \rightarrow D_\frac{r}{2}(z_i) $ is a proper map with degree at most $d$. Since $D_{\frac{2}{3}.\frac{r}{2}}(z_i) \subset D_\frac{r}{2}(z_i)$, Lemma \ref{diameterzero} gives that the diameter of the   component $\tilde{\tilde{V}}$ of $f^{-1}(D_{\frac{r}{3}}(z_i))$, with respect to the hyperbolic distance of $\tilde{ V}$, is at most  $ C(d, \frac{2}{3})$. Since $\tilde{V} \subset \Omega$, $diam_{\Omega} \tilde{\tilde{V}} < C(d,\frac{2}{3})$. Again using Lemma \ref{diameterzero} for $ D_\frac{2r}{3}(w) \subset D_{r} $ and arguing similarly, we have that $diam_\Omega \tilde{ U} < C(d,\frac{2}{3})$  for each component $\tilde{ U}$ of $f^{-1}(D_\frac{2r}{3}(w))$. Note that $D_{r} \subset D_\frac{2r}{3}(w) \cup_{i=1}^{N_0-1} D_\frac{r}{3}(z_i)  $. Since $f:V'_1 \rightarrow D_{r}$ is a proper map with degree at most $d$, the pre-image of each of the above mentioned disks has at most $d$ components. Since the diameter of each such pre-image component with respect to $d_\Omega$ is less than  $C(d,\frac{2}{3})$, $diam_\Omega V'_1 < d N_0 C(d, \frac{2}{3})=C_1$. Thus the claim is proved for $n=1$.
	
\par 
Assume that the claim is true for $n =m$.  This implies that for every connected component $V$ of $f^{-m}(D_{r})$ is simply connected, $deg(f^m:V  \rightarrow D_{r}) \leq d$ and $diam_\Omega V  < C_1$. We shall be done by proving these  for $n=m+1$.
Let $V'_{m+1}$ be a component of $f^{-(m+1)}(D_r)$. If $f(V_{m+1}')=V_{m}'$ then $V_{m}'$ is a component of $f^{-m}(D_r)$, which, by the choice of $\Omega$ and the induction assumption, contains at most one critical value.  It now follows from Lemma \ref{simplyconnectedlema} that $V'_{m+1}$ is simply connected. By the choice of $\Omega$, each critical point $c \in NC$ appears at most once in $U_i =f^i(V'_{m+1})$  for $i=0,~1, \dots, m$. (If not then $c \in U_i \cap U_{i+k}$ and $U_{i+k}$ contains $c$ as well as $f^{k}(c)$, which gives that $diam_{\Omega} U_{i+k} \geq C_1$. But this is contrary to the induction assumption).	 Thus, \begin{equation}
	\label{degree}
	deg(f^{m+1}:V'_{m+1} \rightarrow D_{r}) \leq d.
	\end{equation}
	
Now consider the same open cover $\{D_{\frac{r}{3}}(z_i): i=1,2,3,\cdots, N_{0}-1\} \cup D_{\frac{2r}{3}}$ of $D_r$. Using ~(\ref{degree}) and repeating the arguments as earlier, we get that
	\begin{equation}
	\label{tendingto0}
	diam_\Omega V'_{m+1} < C_1.
	\end{equation}  
	This proves the claim for all $n$.
	\par
	Suppose on the contrary that $J$ is not singleton. Then its spherical diameter  $|J|_s$ is positive. Let $z \in J$. Then $z \neq 0, \infty$ as $J \subset \Omega \subset \widehat{\mathbb{C}}$ by our earlier assumption. Choose a sufficiently small $C>0$ such that $2 C' (e^{2C} -1) \frac{1}{|z|}<|J|_s $ where $C'$ is as mentioned in Lemma~\ref{sphericaldiameter}.  Also, choose $0 < \rho <1$ such that $C(d, \rho) <C$. This is possible as $\lim_{r\to 0}  {C(d,r)}=0$ by Lemma~\ref{diameterzero}. Note that $D_{\rho r} \subset D_r$. Since $f^{n_k}(J) \to w$, there is an $n$ such that $f^{n}(J) \subset D_{\rho r}$. Let  $W'$ and $W$ be the components of $f^{-n}(D_{\rho r})$ and $f^{-n}(D_r)$  respectively, each containing $J$. As already proved, $W'$ and $W$ are simply connected and $\deg(f^{n}: W' \to D_{\rho r}) \leq d$. Therefore, $diam_{W}W' \leq C(d, \rho) <C$. By Lemma~\ref{sphericaldiameter}, $|W'|_s \leq 2 C' (e^{2C} -1) \frac{1}{|z|}$, which is less than $|J|_s$. But $J$ is properly contained in $W'$ as $f^{n}(J)$ is a compact subset of $D_{\rho r}$ and $f^{n}: W' \to D_{\rho r}$ is proper. However, this is not possible as $|J|_s>0$. Therefore, $J$ is singleton, and the proof  completes.  
\end{proof}
Here is a useful remark.

\begin{rem}  Under the hypotheses of the above theorem, for every $\epsilon>0$ and every non-degenerate Julia component  $J$, there is $n_0$ such that $J \subset \{f^{-n}(B)\}_\epsilon$ for all $n > n_0$.  The $\epsilon$-neighborhood of a set $A$, denoted by $A_{\epsilon}$ is defined as  $\cup_{a \in A} D_{\epsilon}(a)$.   Suppose on the contrary, for a  non-degenerate Julia component $J$ and  an $\epsilon>0$, there is a sequence $z_k  \in J  $  and an increasing sequence $n_k$ such that $z_k \notin \{f^{-n_k} (B)\}_\epsilon$. Then consider a limit point $z^*$ of $\{z_k\}_{k >0}$. Since $f^{-n}(B) \supset f^{-n+1}(B)$ for all $n$ by the construction of $B$,  $z_k \notin  \{f^{-j} (B)\}_\epsilon$ for any $ j =1,2,3,\cdots n_k$. It is clear that $z^* \notin  \{f^{-n_k} (B)\}_\epsilon$ for any $k$. In other words $z^*$ is not a limit point of $\cup_{k>0}f^{-n_k}(B)$. Consequently, there is a disk $D$ around $z^*$ such that $\{f^{n_k}\}_{k>0}$  omits all the points of $B$ on $D$. However this is not possible as $z^*$ is in the Julia set.

	\label{shrinking}
\end{rem}
Let $|A|$ denote the   Euclidean diameter of the subset $A$ of $\widehat{\mathbb{C}}$.

\begin{proof}[Proof of Theorem~\ref{rotationaldomainboundary}]
	Let $V$ be a $p$-periodic rotation  domain of $f$.	If $N \subset V$ is an open set then there exists an $\epsilon >0$ such that $ |N_{-n}| \geq  \epsilon$ for all $n$, where $N_{-n}$ is the component of $f^{-n}(N)$ contained in a domain belonging to the cycle containing $V$. To see this,  let $\gamma_1$ and $\gamma_2$ be two $f^p-$ invariant Jordan curves  intersecting $N$. Then $\epsilon= \min_{1 \leq i \leq p} d_{H}(f^i(\gamma_1), f^i(\gamma_2)) >0$ where $d_{H}$ denotes the Hausdorff distance.   Since $N$ intersects $\gamma_1$ and $\gamma_2$, $N_{-n}$ must intersect  $f^{-n} (\gamma_1)$ and $f^{-n}(\gamma_2)$. But  $f^{-n}(\gamma_j)= f^{p-n}(\gamma_j)$ for $j=1,2$. Since $d_{H} (f^i(\gamma_1), f^i(\gamma_2)) \geq \epsilon$ for $1 \leq i \leq p$, 
	
\begin{equation}
|N_{-n}| \geq \epsilon~\mbox{ for all}~n.
\label{nonshrinking}\end{equation}
	
	 Note that the number of critical values not contained in the Fatou component containing the bov is finite. So $f$ has at most finitely many parabolic  periodic cycles.
 \par 
For a Fatou component $V_k$ in the periodic cycle of $V$, let  $z_0 \in \partial V_k$ for some $ 0 \leq k \leq p-1$ such that it is not contained in the $\omega-$limit set of the recurrent critical points. Since the number of parabolic cycles is finite, without loss of generality we asume that $z_0$ is not a parabolic periodic point. Let $D_{2r}$ be a ball centered at $z_0$ contained in $\mathbb{C} \setminus B$, where $B$ is a forward invariant compact set containing all the critical values belonging to the Fatou set ( See the proof of Theorem~\ref{totallydisconnectedness}). Further, let $D_{2r}$ be such that it
 does not intersect the $\omega-$limit set of any recurrent critical point  or any parabolic periodic  point. If $D_{-2n}$ is the component  of  $f^{-np}(D_{2r}) $ intersecting the boundary of $V_k$ for some $k$ and  $D_{-n}$ is the component of  $f^{-np}(D_{r}) $ contained in $D_{-2n}$, then it follows from the proof of Theorem~\ref{totallydisconnectedness} that  there is a $C>0$ with $diam_{D_{-2n}}(D_{-n}) < C$ for all  $n$. Let $z_n \in D_{-n}$ such that $f^{np}(z_n)=z_0$. Since $D_{-2n} \subset f^{-np}(\widehat{\mathbb{C}} \setminus B)$ and by Remark~\ref{shrinking}, $d(z_n, \partial  f^{-np}(\widehat{\mathbb{C}} \setminus B) ) \to 0 $ as $n \to \infty$, it follows from Lemma~\ref{sphericaldiameter} that $|D_{-n}|_s \to 0$ as $n \to \infty$.  
However, this is a contradiction to (\ref{nonshrinking}). 
Thus the boundary of the rotation  domains are contained in the closure of the forward orbit of the recurrent critical points.
\end{proof}
\begin{rem}
	In Theorem~\ref{rotationaldomainboundary}, it is enough to assume that 	the number of critical points corresponding to each critical value belonging to the Julia set is finite. This is evident from the proof.
\end{rem}

\section{ Example}
Now we give a class of functions for which the bov is the only limit point of their critical values.
\par
Let $\alpha \neq 0, \beta \in \mathbb{C}$ and  $P$ be a non-constant polynomial. Since $P(z)+ e^z \to \infty $ as $z \to \infty$, it follows from  Theorem 2.2,~\cite{bov-1}  that $\beta$ is the bov of $f_{\alpha, \beta}(z)=\frac{\alpha}{P(z)+e^z}+\beta$.  Now $f'_{\alpha, \beta} (z)=-\alpha\frac{P'(z)+e^z}{(P(z)+e^z)^2}$. Since the bov is a limit point of the critical values, $f_{\alpha, \beta}$ has infinitely many critical values and hence infinitely many critical points. Let $z_n$ be the set of all critical points. Then  $P'(z_n)+e^{z_n}=0$. Further, every critical point which is a pole must satisfy $P(z)+e^{z}=0$. Therefore   $P(z_n)-P'(z_n)=0$ for each multiple pole $z_n$. Since $P(z)-P'(z)=0$ has at most  $d$ distinct roots,  $f_{\alpha, \beta}(z_n) = \infty$ for at most finitely many values of $n$. Note that $z_n \rightarrow \infty$ as $n \rightarrow \infty$ and  $P(z)-P'(z) \to \infty$ as $z \to \infty$.  This gives that $P(z_n)-P'(z_n) \rightarrow \infty$ as $z_n \rightarrow \infty$. Thus $f_{\alpha, \beta}(z_n)=\frac{\alpha}{P(z_n)+e^{z_n}}+ \beta =\frac{\alpha}{P(z_n)-P'(z_n)}+ \beta  \rightarrow \beta$ as $n \rightarrow \infty$. Therefore $\beta$  is the only limit point of the critical values of $f_{\alpha, \beta}$.
Further, $f_{\alpha, \beta}$ has invariant attracting domain, invariant parabolic domain and invariant Siegel disks for suitable choices of $\alpha$, $\beta$ and $P$~\cite{bov-1}.

\end{document}